\documentclass[12pt, reqno]{amsart}
\usepackage{floatrow}
\usepackage{amssymb, amsfonts, amsthm, amsmath, amscd}
\usepackage{graphicx}   
\usepackage[latin2]{inputenc}
\usepackage{t1enc}
\usepackage[mathscr]{eucal}
\usepackage{indentfirst}
\usepackage{pict2e}
\usepackage{epic}
\numberwithin{equation}{section}
\usepackage{stmaryrd}
\usepackage{scalerel,stackengine}
\usepackage[margin=2.5cm]{geometry}
\usepackage{mathtools}
\usepackage[colorlinks = true,
            linkcolor = blue,
            urlcolor  = blue,
            citecolor = blue,
            anchorcolor = blue]{hyperref}
\usepackage[titletoc]{appendix}
\usepackage[T1]{fontenc}
\usepackage[numbers,square]{natbib}
\usepackage{adjustbox}
\usepackage{dsfont}
\usepackage{times}
\usepackage{enumerate}
\usepackage{setspace}
\usepackage{leftidx}
\usepackage{multirow}

\usepackage{tikz-cd}
\usepackage{pgf,tikz}
\usetikzlibrary{math}
\usetikzlibrary{arrows,decorations.markings}
\usetikzlibrary{patterns,intersections}
\usetikzlibrary{calc}
\usetikzlibrary{decorations.pathreplacing}

\makeatletter
\DeclareRobustCommand\bigop[1]{%
  \mathop{\vphantom{\sum}\mathpalette\bigop@{#1}}\slimits@
}
\newcommand{\bigop@}[2]{%
  \vcenter{%
    \sbox\z@{$#1\sum$}%
    \hbox{\resizebox{\ifx#1\displaystyle.9\fi\dimexpr\ht\z@+\dp\z@}{!}{$\m@th#2$}}%
  }%
}
\makeatother
\newenvironment{myproof}[2]{\paragraph{\textit{Proof of {#1} }{#2}.}}{\hfill$\square$}

\theoremstyle{plain}
\newtheorem{theorem}{Theorem}[section]
\newtheorem{lemma}[theorem]{Lemma}

\newtheorem{proposition}[theorem]{Proposition}

\theoremstyle{definition}
\newtheorem{definition}[theorem]{Definition}

\newtheorem{remark}[theorem]{Remark}

\def\CC{\mathbb{C}}

\def\QQ{\mathbb{Q}}
\def\RR{\mathbb{R}}
\def\ZZ{\mathbb{Z}}
\def\PP{\mathbb{P}}
\def\ds{\displaystyle}

\def\a{\alpha}
\def\b{\beta}
\def\g{\gamma}
\def\s{\sigma}

\def\G{\Gamma}
\def\w{\omega}
\def\l{\lambda}
\def\ra{\rightarrow}

\def\mb{\mathbf}
\def\O{\Omega}

\def\LL{\mathcal{L}}

\def\BB{\mathcal{B}}

\def\FF{\mathcal{F}}

\def\MM{\overline{\mathcal{M}}}
\def\OO{\mathcal{O}}

\def\gg{\mathfrak{g}}
\def\hh{\mathfrak{t}}

\def\ol{\overline}

\def\scl{\sigma}
\def\ascl{\xi}
\newcommand{\localb}[1]{[t^{#1}]}

\def\Q{Q^{\vee}}

\newcommand{\fib}[1]{\mathcal{E}_{#1}(G/P)}

\def\mr{\mathring}

\def\wl{wt_{\lambda}}
\def\degg{c}
\def\MMw{\MM(\wl,\eta)}
\def\MMm{\MM(\mu,\eta)}
\def\ag{\mathcal{G}r}
\def\bl{\bullet}
\def\AA{\mathbb{A}}
\def\vp{\varphi}
\def\fibb{\mathcal{E}}
\def\wll{w(\lambda)}
\def\TT{\mathcal{T}}

\DeclareMathOperator{\ec}{Eff}
\DeclareMathOperator{\ecb}{Eff}
\DeclareMathOperator{\lie}{Lie}
\DeclareMathOperator{\fof}{Frac}
\DeclareMathOperator{\pr}{pr}

\DeclareMathOperator{\ev}{ev}
\DeclareMathOperator{\id}{id}

\DeclareMathOperator{\pd}{PD}

\DeclareMathOperator{\coker}{coker}
\DeclareMathOperator{\Pic}{Pic}
\DeclareMathOperator{\spec}{Spec}
\DeclareMathOperator{\eend}{End}
\DeclareMathOperator{\ehom}{Hom}
\DeclareMathOperator{\pt}{pt}
\DeclareMathOperator{\vdim}{vdim}

\def\RRR{H_T^{\bl}(\pt)}
\makeatletter
\newcommand*\bulletsmall{\mathpalette\bulletsmall@{.5}}
\newcommand*\bulletsmall@[2]{\mathbin{\vcenter{\hbox{\scalebox{#2}{$\m@th#1\bullet$}}}}}
\makeatother

\begin{document}
\title[PLS's theorem is an affine analogue of quantum Chevalley formula]{Peterson-Lam-Shimozono's theorem is an affine analogue of quantum Chevalley formula}

\author{Chi Hong Chow}

\begin{abstract} 
We give a new proof of an unpublished result of Dale Peterson, proved by Lam and Shimozono, which identifies explicitly the structure constants, with respect to the quantum Schubert basis, for the $T$-equivariant quantum cohomology $QH^{\bullet}_T(G/P)$ of any flag variety $G/P$ with the structure constants, with respect to the affine Schubert basis, for the $T$-equivariant Pontryagin homology $H^T_{\bullet}(\mathcal{G}r)$ of the affine Grassmannian $\mathcal{G}r$ of $G$, where $G$ is any simple simply-connected complex algebraic group. 

Our approach is to construct an $H_T^{\bullet}(pt)$-algebra homomorphism by Gromov-Witten theory and show that it is equal to Peterson's map. More precisely, the map is defined via Savelyev's generalized Seidel representations which can be interpreted as certain Gromov-Witten invariants with input $H^T_{\bullet}(\mathcal{G}r)\otimes QH_T^{\bullet}(G/P)$. We determine these invariants completely, in a way similar to how Fulton and Woodward did in their proof of the quantum Chevalley formula.
\end{abstract}

%\subjclass[2010]{} 
%\keywords{} 

%based loop group, Bott-Samelson variety, moment correspondence
\maketitle

%%%%%%%%%%%%%%%%%%%%%%%%%%%%%%%%%%%
%%%%%%%%%%%%%%%%%%%%%%%%%%%%%%%%%%%
%%%%%%%%%%%%%%%%%%%%%%%%%%%%%%%%%%%
%%%%%%%%%%%%%%%%%%%%%%%%%%%%%%%%%%%
%%%%%%%%%%%%%%%%%%%%%%%%%%%%%%%%%%%
%%%%%%%%%%%%%%%%%%%%%%%%%%%%%%%%%%%
\section{introduction}\label{1}
Let $G$ be a simple simply-connected complex algebraic group. The quantum (resp. affine) Schubert calculus studies the algebra structure on the $T$-equivariant quantum cohomology $QH^{\bl}_T(G/B)$ of the complete flag variety $G/B$ (resp. the $T$-equivariant Pontryagin homology $H^T_{\bl}(\ag)$ of the affine Grassmannian $\ag$ of $G$) in terms of the quantum Schubert classes $\{q^{\b}\scl_v\}_{(\b,v)\in \ecb\times W}$ (resp. the affine Schubert classes $\{\ascl_{\wl}\}_{\wl\in W_{af}^-}$). An unpublished result of Dale Peterson \cite{Peter}, announced during the lectures he gave at MIT in 1997, states that these two calculi are equivalent:
\begin{theorem} \label{main} The map
\[
\begin{array}{ccccc}
\Phi&:&H^T_{-\bl}(\ag)&\ra &QH^{\bl}_T(G/B)[q_i^{-1}|~i\in I]\\ [.5em]
& &  \ascl_{\wl} & \mapsto & q^{\l}\scl_{w}
\end{array}
\]
is a graded homomorphism of $\RRR$-algebras.
\end{theorem}
\noindent A published proof, given by Lam and Shimozono \cite{LS}, is algebraic and combinatorial. In this paper, we present a geometric proof by taking $\Phi$ to be the algebro-geometric and $T$-equivariant version of a map constructed by Savelyev \cite{S_QCC} who generalized Seidel representations \cite{Seidel} from 0-cycles in $\ag$ to higher dimensional ones, and showing this map to have the desired form.

In the same paper, Lam and Shimozono also proved
\begin{theorem}\label{P}
A parabolic version of Theorem \ref{main} holds.
\end{theorem}
\noindent We will prove Theorem \ref{P} as well. Since even stating it requires a substantial number of Lie-theoretic notations, we postpone the statement to Section \ref{4d} where we prove the Borel and parabolic cases simultaneously. 

\begin{remark}\label{intrormk0} Savelyev has already computed his map partially. In \cite{S_Bott}, he showed that his map defined for $\PP^n$ is non-zero on each generator of $\pi_*(\O SU(n+1))\otimes\QQ$ which has degree $<2n$. In \cite{S_JDG}, he proved that for any $\wl\in W_{af}^-$ such that $w$ is the longest element of $W$, his map defined for $G/B$ sends $\ascl_{\wl}$ to $q^{\l}\scl_{w}$ plus some higher terms with respect to an \text{action functional} on the space of sections of Hamiltonian fibrations.
\end{remark}

\begin{remark}\label{intrormk1} The proof of Theorem \ref{main} given by Lam and Shimozono relies on the equivariant quantum Chevalley formula \cite{M} which is the $T$-equivariant generalization of another unpublished result of Peterson proved by Fulton and Woodward \cite{FW}. Although we do not apply this formula directly, we do apply the key idea of the proof: the transverse property between the Schubert cells and the opposite Schubert cells in $G/P$ which implies that the moduli spaces for all two-pointed Gromov-Witten invariants are simultaneously regular and $T$-equivariant, allowing us to count the elements of their zero-dimensional components easily.
\end{remark}

\begin{remark}\label{intrormk2} Unlike the proof by Lam and Shimozono, our proof of Theorem \ref{P} is independent of Peterson-Woodward's comparison formula \cite{Woodward} which expresses explicitly the Schubert structure constants for $QH^{\bl}(G/P)$ in terms of those for $QH^{\bl}(G/B)$. In fact, our work provides an alternative proof of this formula because it can be derived directly from Theorem \ref{P} as shown by Huang and Li \cite[Proposition 2.10]{HL}.
\end{remark}

%%%%%%%%%%%%%%%%%%%%%%%%%%%%%%%%%%%
%%%%%%%%%%%%%%%%%%%%%%%%%%%%%%%%%%%
%%%%%%%%%%%%%%%%%%%%%%%%%%%%%%%%%%%
%%%%%%%%%%%%%%%%%%%%%%%%%%%%%%%%%%%
%%%%%%%%%%%%%%%%%%%%%%%%%%%%%%%%%%%
%%%%%%%%%%%%%%%%%%%%%%%%%%%%%%%%%%%
\section*{Acknowledgements}
The first version of this paper, which already contains all key ideas, was written when the author was a PhD student at the Chinese University of Hong Kong. He would like to thank the referees for useful comments which help to improve the exposition substantially.
%%%%%%%%%%%%%%%%%%%%%%%%%%%%%%%%%%%
%%%%%%%%%%%%%%%%%%%%%%%%%%%%%%%%%%%
%%%%%%%%%%%%%%%%%%%%%%%%%%%%%%%%%%%
%%%%%%%%%%%%%%%%%%%%%%%%%%%%%%%%%%%
%%%%%%%%%%%%%%%%%%%%%%%%%%%%%%%%%%%
%%%%%%%%%%%%%%%%%%%%%%%%%%%%%%%%%%%
\section{Preliminaries}\label{2}
\subsection{Some notations} \label{2a}
Let $G$ be a simple simply-connected complex algebraic group and $T\subset G$ a maximal torus. Put $\gg:=\lie(G)$ and $\hh:=\lie(T)$. Denote by $R$ the set of roots associated to the pair $(\gg,\hh)$. We have the root space decomposition
\[\gg=\hh\oplus\bigoplus_{\a\in R}\gg_{\a}\]
where each $\gg_{\a}$ is a one-dimensional eigenspace with respect to the adjoint action of $\hh$. Denote by $W$ the Weyl group. Fix a fundamental system $\{\a_i\}_{i\in I}$ of $R$, where $I:=\{1,\ldots,r\}$. Denote by $R^+\subset R$ the set of positive roots spanned by the $\a_i$'s. We have two particular Borel subgroups $B^-$ and $B^+$ of $G$ containing $T$ with their Lie algebras equal to $\hh\oplus \bigoplus_{\a\in -R^+}\gg_{\a}$ and $\hh\oplus\bigoplus_{\a\in R^+}\gg_{\a}$, respectively.

Let $W_{af}:=W\ltimes\Q$ be the affine Weyl group  where $\Q:=\sum_{\a\in R}\ZZ\cdot\a^{\vee}\subset\hh$ is the lattice spanned by the coroots. Elements of $W_{af}$ are denoted by $\wl$ with $w\in W$ and $\l\in\Q$ (where $t_{\l}$ means the translation $x\mapsto x+\l$). Denote by $W_{af}^-$ the set of minimal length coset representatives in $W_{af}/W$. It is easy to see that the map $W_{af}^-\ra \Q$ defined by $\wl\mapsto w(\l)$ is bijective. 
%%%%%%%%%%%%%%%%%%%%%%%%%%%%%%%%%%%
%%%%%%%%%%%%%%%%%%%%%%%%%%%%%%%%%%%
\subsection{Flag varieties} \label{2b} Let $P$ be a parabolic subgroup of $G$ containing $B^+$. Define $R^+_P\subseteq R^+$ to be the subset such that
\[ \lie(P)=\lie(B^+)\oplus\bigoplus_{\a\in -R_P^+}\gg_{\a}\]
and $R_P:=R^+_P\cup (-R^+_P)$. Let $I_P\subseteq I$ be the set of $i\in I$ such that $\a_i\in R_P^+$. Define $\Q_P:=\sum_{\a\in R^+_P}\ZZ\cdot\a^{\vee}\subseteq\Q$. Denote by $W_P\subseteq W$ the subgroup generated by the simple reflections $s_{\a_i}$ with $i\in I_P$ and by $W^P$ the set of minimal length coset representatives in $W/W_P$. For any $v\in W^P$, define $y_v:=\dot{v}P\in G/P$ where $\dot{v}\in N(T)$ is any representative of $v$. Then $\{y_v\}_{v\in W^P}$ is the set of $T$-fixed points of $G/P$.

The $B^-$-orbits $B^{-}\cdot y_v\subseteq G/P,~v\in W^P$ are called the Schubert cells and the $B^+$-orbits $B^+\cdot y_v\subseteq G/P,~v\in W^P$ are called the opposite Schubert cells. Define the (opposite) Schubert classes
\begin{align*}
\scl_v:=&~ \pd \left[~\ol{B^{-}\cdot y_v}~\right]\in H_T^{2\ell(v)}(G/P)\\
\scl^v:=&~ \pd \left[~\ol{B^{+}\cdot y_v}~\right]\in H_T^{\dim_{\RR}(G/P)-2\ell(v)}(G/P).
\end{align*}
Then $\{\scl_v\}_{v\in W^P}$ and $\{\scl^v\}_{v\in W^P}$ are $H_T^{\bl}(\pt)$-bases of $H_T^{\bl}(G/P)$. 

\vspace{.5cm}
The following well-known fact is crucial to us.
\begin{lemma}\label{2blemma} Every Schubert cell intersects every opposite Schubert cell transversely. In particular, $\{\scl_v\}_{v\in W^P}$ and $\{\scl^v\}_{v\in W^P}$ are dual to each other with respect to $\int_{G/P}-\cup -$.
\end{lemma}
\begin{proof}
See e.g. \cite[Section 7]{FW}.
\end{proof}

It is also well-known that the closures of (resp. opposite) Schubert cells have $B^-$-equivariant (resp. $B^+$-equivariant) resolutions, e.g. the Bott-Samelson-Demazure-Hansen resolutions. See e.g. \cite[Section 2]{Brion} for the construction.
\begin{definition}\label{2bdef}
For each $v\in W^P$, fix a $B^+$-equivariant morphism
 \[ f_{G/P,v}:\G_v\ra G/P\]
which is the composition of a resolution $\G_v\ra\ol{B^+\cdot y_v}$ and the inclusion $\ol{B^+\cdot y_v}\hookrightarrow G/P$.
\end{definition}

We now recall the $T$-equivariant quantum cohomology of $G/P$. See e.g. \cite{Mirror, FP, KM} for more details. There are isomorphisms
\begin{equation}\label{2bisom}
H_2(G/P)\simeq\Q/\Q_P\simeq \bigoplus_{i=1}^k\ZZ\cdot\a^{\vee}_i
\end{equation}
where the first is defined as the dual of the composition of three isomorphisms:
\[ \left(\Q/\Q_P\right)^*\xrightarrow{\rho~\mapsto L_{\rho}} \Pic(G/P) \xrightarrow{c_1} H^2(G/P)\simeq H_2(G/P)^*. \]
Here, $L_{\rho}$ is the line bundle $G\times_P \CC_{-\rho}$ where $\CC_{-\rho}$ is the one-dimensional representation of weight $-\rho$ on which $P$ acts by forgetting the semi-simple and unipotent parts. Denote by $\ec\subset H_2(G/P)$ the semi-group of effective curve classes in $G/P$.  Under \eqref{2bisom}, $\ec$ corresponds to the semi-subgroup of $\Q/\Q_P$ generated by $\a^{\vee}_i$ with $i\in I\setminus I_P$.

Define the $T$-equivariant quantum cohomology of $G/P$
\[ QH_T^{\bl}(G/P):= H_T^{\bl}(G/P)\otimes\ZZ[q_i|~i\in I\setminus I_P].\] 
We grade $QH_T^{\bl}(G/P)$ by declaring each $q_i$ to have degree $2\sum_{\a\in R^+\setminus R^+_P}\a(\a^{\vee}_i)$. The $T$-equivariant quantum cup product $\star$ is a deformation of the $T$-equivariant cup product, defined by
\[\s_u\star \s_v := \sum_{w\in W^P}\sum_{\mb{d}} \left(\prod_{i\in I\setminus I_P} q_i^{d_i}\right)  \left(\int_{\MM_{0,3}(G/P,\beta_{\mb{d}})} \ev_1^*\s_u\cup \ev_2^*\s_v\cup  \ev_3^*\s^w \right) \s_w,\]  
where
\begin{enumerate}[(i)]
\item $\mb{d}=\{d_i\}_{i\in I\setminus I_P}$ runs over the set of $(I\setminus I_P)$-tuples of non-negative integers;

\item $\beta_{\mb{d}}\in \ec$ corresponds to $\sum_{i\in I\setminus I_P}d_i\a^{\vee}_i$ via the isomorphism \eqref{2bisom};

\item $\MM_{0,3}(G/P,\beta_{\mb{d}})$ is the moduli of genus-zero stable maps to $G/P$ of degree $\beta_{\mb{d}}$ with three marked points and 
\[\ev_1,\ev_2,\ev_3:\MM_{0,3}(G/P,\beta_{\mb{d}})\ra G/P\]
are the evaluation morphisms at these marked points; and

\item the integral $\int_{\MM_{0,3}(G/P,\beta_{\mb{d}})}$ is the $T$-equivariant integral.
\end{enumerate}

\noindent Then $(QH_T^{\bl}(G/P),\star)$ is a graded commutative $\RRR$-algebra. 

%%%%%%%%%%%%%%%%%%%%%%%%%%%%%%%%%%%
%%%%%%%%%%%%%%%%%%%%%%%%%%%%%%%%%%%
\subsection{Affine Grassmannian} \label{2c}
The affine Grassmannian $\ag$ of $G$ is by definition (see e.g. \cite[Section 1.2]{Zhu}) the functor
\[
\begin{array}{cccc}
& AffSch_{\CC} &\ra& Sets\\ [0.5em]
&\spec R&\mapsto & \left\{\begin{array}{l}
\text{isomorphism classes of }(\fibb^o,\nu^o)\text{ where}\\[.2em]
\fibb^o\text{ is a }G\text{-torsor over }\spec R[[z]], \\[.2em]
\nu^o:\fibb^o|_{\spec R((z))}\xrightarrow{\sim}\spec R((z))\times G\\[.2em]
\text{is a trivialization} 
\end{array} \right\}
\end{array}.
\]
(We use the notation $(\fibb^o,\nu^o)$ instead of a more natural one $(\fibb,\nu)$ because the latter is reserved for $G$-torsors over $\PP^1$.) It is well-known that $\ag$ is represented by an Ind-projective Ind-scheme. See e.g. \cite[Theorem 1.2.2]{Zhu}. By Beauville-Laszlo's theorem \cite[]{BL}, $\ag$ also represents the subfunctor
\[
\spec R\mapsto \left\{\begin{array}{l}
\text{isomorphism classes of }(\fibb^o,\nu^o)\text{ where}\\[.2em]
\fibb^o\text{ is a }G\text{-torsor over }\spec R[z], \\[.2em]
\nu^o:\fibb^o|_{\spec R[z,z^{-1}]}\xrightarrow{\sim}\spec R[z,z^{-1}]\times G\\[.2em]
\text{is a trivialization} 
\end{array} \right\}.
\]

Let $\mu\in\Q$. Denote by $t^{\mu}$ the $\spec\CC$-point of $\ag$ represented by the trivial $G$-torsor with trivialization $(z,g)\mapsto (z,\mu(z)g)$. One checks easily that it is a $T$-fixed point of $\ag$. It is known that if $G$ is simply-connected, every $T$-fixed point of $\ag$ is of this form. Define $\BB:=\ev_{z=0}^{-1}(B^-)$ where $\ev_{z=0}:G(\CC[[z]])\ra G$ is the evaluation map at $z=0$. For any $\mu\in\Q$, the orbit $\BB\cdot t^{\mu}$ is isomorphic to an affine space and we call it an affine Schubert cell. In this paper, it is more convenient to index affine Schubert cells by $W_{af}^-$ instead of $\Q$. (See Section \ref{2a} for the definition of $W_{af}^-$.) For any $\wl\in W_{af}^-$, we define the affine Schubert class
\[ \ascl_{\wl}:= \left[~\ol{\BB\cdot t^{\wll}}~\right]\in H^T_{2\ell(\wl)}(\ag).\]
Then $\{\ascl_{\wl}\}_{\wl\in W_{af}^-}$ is an $\RRR$-basis of the $T$-equivariant homology $H^T_{\bl}(\ag)$ of $\ag$.

Denote by $\LL G$ the loop group functor $\spec R\mapsto G(R((z)))$. We have a natural group action 
\[ \LL G\times\ag\ra \ag\]
defined by 
\begin{equation}\label{2caction}
\vp\cdot (\fibb^o,\nu^o):= (\fibb^o,\vp\cdot \nu^o)
\end{equation}
for any $\vp\in G(R((z)))$ and $\spec R$-point $[(\fibb^o,\nu^o)]$ of $\ag$, where $\vp\cdot\nu^o(p):=(x,\vp(z)g)\in\spec R((z))\times G$ for $\nu^o(p)=(x,g)$.

For any $\a\in R$ and $k\in\ZZ$, denote by $U_{\a,k}\subset\LL G$ the affine root group $\exp(z^k\gg_{\a})\simeq\mathbb{G}_a$.
\begin{definition}\label{2cgood}
Let $H$ be a subgroup of $G$. A morphism $f:\G\ra\ag$ from a variety $\G$ to $\ag$ is said to be $H$-good if $\G$ has an algebraic $H$-action and an algebraic $U_{\a,k}$-action for each $\a\in R$ and $k>0$ such that $f$ is equivariant with respect to these group actions. 
\end{definition}

\begin{lemma}\label{2clemma}
For any $\wl\in W_{af}^-$, there exists a $B^-$-good morphism $f:\G\ra\ag$ which is the composition of a resolution $\G\ra \ol{\BB\cdot t^{\wll}}$ and the inclusion $\ol{\BB\cdot t^{\wll}}\hookrightarrow \ag$.
\end{lemma} 
\begin{proof}
Notice that for $\mathcal{H}=B^-$ or $U_{\a,k}$ with $\a\in R$ and $k>0$, the $\mathcal{H}$-action on $\ag$ induces an $\mathcal{H}$-action on $\ol{\BB\cdot t^{\wll}}$ such that the inclusion $\iota:\ol{\BB\cdot t^{\wll}}\hookrightarrow \ag$ is $\mathcal{H}$-equivariant.

By \cite[Proposition 3.9.1 \& Theorem 3.26]{equivresolution}, there exists a resolution $r:\G\ra\ol{\BB\cdot t^{\wll}}$ such that every algebraic group action on $\ol{\BB\cdot t^{\wll}}$ lifts to an algebraic group action on $\G$. It follows that the composition $f:=\iota\circ r$ is a $B^-$-good morphism.

Alternatively, one can take a Bott-Samelson-Demazure-Hansen resolution of $\ol{\BB\cdot t^{\wll}}$. See e.g. \cite[Section 8]{BottSameslon} for the construction.
\end{proof}
\begin{definition}\label{2cdef}
For each $\wl\in W_{af}^-$, fix a $B^-$-good morphism \[ f_{\ag,\wl}:\G_{\wl}\ra \ag\] 
which is the composition of a resolution $\G_{\wl}\ra \ol{\BB\cdot t^{\wll}}$ and the inclusion $\ol{\BB\cdot t^{\wll}}\hookrightarrow \ag$.
\end{definition}

Let $K$ be a maximal compact subgroup of $G$ such that $T_K:=T\cap K$ is a maximal torus of $K$. Let $\O_{pol} K$ be the space of polynomial based loops in $K$. It is well-known that the canonical map $\O_{pol} K\ra \ag$ is a $T_K$-equivariant homeomorphism. See \cite[Theorem 1.6.1]{Zhu} for an exposition of the proof of this result and the references cited therein, namely \cite[Section 4]{Nadler} and \cite[Section 8.3]{Segal}. Notice that $\O_{pol} K$ is a group. Its group structure thus induces an $\RRR$-algebra structure on $H^T_{\bl}(\ag)$. It is called the Pontryagin product. By definition, we have
\begin{equation}\label{2cprod}
[t^{\mu_1}]\bulletsmall [t^{\mu_2}] = [t^{\mu_1+\mu_2}]\quad \text{ for any }\mu_1,\mu_2\in\Q.
\end{equation}
Since $\{[t^{\mu}]\}_{\mu\in\Q}$ is a basis of $H^T_{\bl}(\ag)\otimes_{\RRR}\fof(\RRR)$, these equalities determine the Pontryagin product completely.
%%%%%%%%%%%%%%%%%%%%%%%%%%%%%%%%%%%
%%%%%%%%%%%%%%%%%%%%%%%%%%%%%%%%%%%
%%%%%%%%%%%%%%%%%%%%%%%%%%%%%%%%%%%
%%%%%%%%%%%%%%%%%%%%%%%%%%%%%%%%%%%
%%%%%%%%%%%%%%%%%%%%%%%%%%%%%%%%%%%
%%%%%%%%%%%%%%%%%%%%%%%%%%%%%%%%%%%
\section{The Savelyev-Seidel homomorphism}\label{3}
%nobadbox
\subsection{$G/P$-bundles} \label{3a}
Let $f:\G\ra\ag$ be a morphism where $\G$ is a variety. It is represented by a pair $(\fibb_f^o,\nu_f^o)$ where $\fibb_f^o$ is a $G$-torsor over $\AA^1_z\times\G$ and $\nu_f^o:\fibb_f^o|_{(\AA^1_z\setminus 0)\times \G}\xrightarrow{\sim} (\AA^1_z\setminus 0)\times \G\times G$ is a trivialization. To see this, take a covering $\{U_i\}_i$ of $\G$ by affine open subsets. By the definition of $\ag$, each $f|_{U_i}$ is represented by a pair $(\fibb^o_{f|_{U_i}},\nu^o_{f|_{U_i}})$. Since in general every pair $(\fibb^o,\nu^o)$ has no non-trivial automorphism (essentially due to the trivialization $\nu^o$), it follows that we can glue $(\fibb^o_{f|_{U_i}},\nu^o_{f|_{U_i}})$ to form the desired pair $(\fibb_f^o,\nu_f^o)$.

Identify $\AA^1_z$ with $\PP^1\setminus \infty$. Glue $\fibb_f^o$ and $(\PP^1\setminus 0)\times\G\times G$ using $\nu_f^o$. The resulting variety is a $G$-torsor over $\PP^1\times \G$ with a trivialization over $(\PP^1\setminus 0)\times\G$. We denote the $G$-torsor by $\fibb_f$ and the trivialization by $\nu_f$.

\begin{remark}
There is a parallel story in the analytic category. In \cite{Segal}, Pressley and Segal defined $\ag$ to be the based loop group $\O K$ with respect to various topologies and showed \cite[Theorem 8.10.2]{Segal} that there is a bijective correspondence between the set of holomorphic maps $f:\G\ra\ag$ and the set of  isomorphism classes of holomorphic principal $G$-bundles over $\PP^1\times\G$ with trivializations over $(\PP^1\setminus \{|z|\leqslant 1\})\times \G$.
\end{remark}

\begin{lemma}\label{3avar}
The associated fiber bundle
\[ \fib{f}:=\fibb_f\times_G G/P\]
exists as a variety. The canonical projection
\[ \pi_f: \fib{f}\ra\PP^1\times\G\]
is smooth and projective. In particular, $\fib{f}$ is smooth (resp. projective) if $\G$ is.
\end{lemma}
\begin{proof}
By the existence of a $G$-linearized ample line bundle on $G/P$ and the descent theory for quasi-coherent sheaves, $\fib{f}$ exists as a scheme. In fact, it is a closed subscheme of a projective bundle over $\PP^1\times\G$. In particular, $\fib{f}$ is separated and of finite type over $\CC$, and $\pi_f$ is projective. Observe that $\fib{f}$ becomes a trivial $G/P$-bundle after a faithfully flat base change. This implies that $\fib{f}$ is reduced, as it is the flat image of a reduced scheme; and that $\pi_f$ is smooth, as it becomes so after a faithfully flat base change. Finally, $\fib{f}$ is irreducible because $\pi_f$ is smooth and has irreducible base and geometric fibers.
\end{proof}

Let $\rho\in (\Q/\Q_P)^*$. Recall the $G$-linearized line bundle $L_{\rho}:=G\times_P\CC_{-\rho}$ on $G/P$. Let $\pr_2:\fibb_f\times G/P\ra G/P$ denote the canonical projection. Then $\pr_2^*L_{\rho}$ is naturally a $G$-linearized line bundle on $\fibb_f\times G/P$ with respect to the diagonal $G$-action, and hence it descends to a line bundle on $\fib{f}$ which we denote by $\LL_{\rho}$. It has a property that its restriction to every fiber of $\pi_f$ is isomorphic to $L_{\rho}$.
\begin{definition} \label{3asection} 
We call $\b\in H_2(\fib{f})$ a section class of $\fib{f}$ if $(\pi_f)_*\b= [\PP^1\times \g_0]$ for some $\g_0\in\G$.
\end{definition}

\begin{definition} \label{3afunction} 
Define a function 
\[ \degg: \{\text{section classes of }\fib{f}\} \ra \Q/\Q_P\]
characterized by the property that for any $\rho\in (\Q/\Q_P)^*$,
\[ \langle \b,c_1(\LL_{\rho})\rangle  = \langle \degg(\b), \rho\rangle.\]
\end{definition}
 
Let $H$ be a subgroup of $G$. Suppose $\G$ has an $H$-action. Then we have an obvious $H$-action on $(\PP^1\setminus 0)\times \G\times G$:
\begin{equation}\label{3aaction}
h\cdot (z,\g,g):= (z,h\cdot\g, hg).
\end{equation}
\begin{lemma}\label{3alemma}
Suppose $f$ is $H$-equivariant. Then the $H$-action on $\fibb_{f}|_{(\PP^1\setminus 0)\times \G}$ defined by \eqref{3aaction} via $\nu_f$ extends to $\fibb_f$ and hence defines an $H$-action on $\fib{f}$.
\end{lemma}
\begin{proof}
Denote by
\[ a:H\times \G\ra \G\quad\text{ and }\quad \pr_{\G}:H\times \G\ra \G\]
the action morphism and the canonical projection respectively. Consider the following two $H\times\G$-points of $\ag$:
\[ \left( (\id_{\AA^1_z}\times a)^*\fibb_f^o, (\id_{\AA^1_z}\times a)^*\nu_f^o\right)\quad\text{ and }\quad\left( (\id_{\AA^1_z}\times \pr_{\G})^*\fibb_f^o, \widetilde{\nu}_f^o\right) \]
where $\widetilde{\nu}_f^o$ is the trivialization of $(\id_{\AA^1_z}\times \pr_{\G})^*\fibb_f^o|_{(\AA^1_z\setminus 0)\times H\times \G}\simeq H\times \fibb_f^o|_{(\AA^1_z\setminus 0)\times \G}$ defined by 
\[ \widetilde{\nu}_f^o(h,p):=(h,z,\g,hg)\quad\text{for }~\nu_f^o(p)=(z,\g,g).\]
By the assumption that $f$ is $H$-equivariant, these two $H\times \G$-points are equal. In other words, there exists an isomorphism between the underlying $G$-torsors which is compatible with the underlying trivializations. Therefore, the composition
\[ H\times \fibb_f^o\simeq (\id_{\AA_z^1}\times\pr_{\G})^*\fibb_f^o\xrightarrow{\sim} (\id_{\AA^1_z}\times a)^*\fibb_f^o\ra \fibb_f^o\]
gives the desired extension, where the last arrow is the canonical projection.
\end{proof}
%%%%%%%%%%%%%%%%%%%%%%%%%%%%%%%%%%%
%%%%%%%%%%%%%%%%%%%%%%%%%%%%%%%%%%%
\subsection{Moduli of sections} \label{3b}
Let $f:\G\ra \ag$ be a morphism where $\G$ is a smooth projective variety. Then $\fib{f}$ is a smooth projective variety by Lemma \ref{3avar}. The subvariety $D_{f,\infty}:=\pi_f^{-1}(\infty\times \G)$ is a smooth divisor of $\fib{f}$ and is identified with $\G\times G/P$ via the trivialization $\nu_f$. Denote by $\iota_{f,\infty}:D_{f,\infty}\hookrightarrow \fib{f}$ the inclusion. 

\begin{definition} \label{3bmodulif} Let $\eta\in\Q/\Q_P$.
\begin{enumerate}
\item Define 
\[\MM(f,\eta):= \bigcup_{\b} ~\MM_{0,1}(\fib{f},\b)\times_{(\ev_1,\iota_{f,\infty})} D_{f,\infty}\]
where $\b$ runs over all section classes of $\fib{f}$ such that $\degg(\b)=\eta$ ($\degg$ is defined in Definition \ref{3afunction}).

\item Define 
\[\ev:\MM(f,\eta)\ra G/P\]
to be the composition
\[ \MM(f,\eta)\ra D_{f,\infty}\simeq \G\times G/P \ra G/P\]
of the morphism induced by $\ev_1$, the isomorphism induced by $\nu_f$ and the canonical projection.
\end{enumerate}
\end{definition}

\begin{lemma}\label{3bdim}
The virtual dimension of $\MM(f,\eta)$ is equal to $\dim\G +\dim G/P +\sum_{\a\in R^+\setminus R^+_P}\a(\eta)$.
\end{lemma}
\begin{proof}
Denote by $\vdim\MM(f,\eta)$ the virtual dimension of $\MM(f,\eta)$. Let $\b$ be a section class of $\fib{f}$ such that $\degg(\b)=\eta$. We have
\begin{equation}\label{3bdimeq1}
\vdim\MM(f,\eta)= \dim\fib{f} +\langle\b,c_1(\TT_{\fib{f}})\rangle -3.
\end{equation}
Since $\fib{f}$ is a $G/P$-bundle over $\PP^1\times\G$ and $\b$ is a section class, the equality \eqref{3bdimeq1} can be simplified to 
\[\vdim\MM(f,\eta)=\dim\G +\dim G/P +\langle\b,c_1(\TT^{vert}_{\pi_f})\rangle \]
where $\TT^{vert}_{\pi_f}$ is the vertical tangent bundle of $\pi_f$. 

It remains to show $\langle\b,c_1(\TT^{vert}_{\pi_f})\rangle=\sum_{\a\in R^+\setminus R_P^+}\a(\eta)$. Recall $\bigwedge^{top}\TT_{G/P}\simeq L_{\rho_P}$ as $G$-linearized line bundles, where $\rho_P:=\sum_{\a\in R^+\setminus R_P^+}\a$. It follows that $\bigwedge^{top}\TT^{vert}_{\pi_f}\simeq \LL_{\rho_P}$, and hence 
\[\langle\b,c_1(\TT^{vert}_{\pi_f})\rangle=\langle\b,c_1({\textstyle \bigwedge^{top}}\TT^{vert}_{\pi_f})\rangle=\langle\b,c_1(\LL_{\rho_P})\rangle=\langle\eta,\rho_P\rangle=\sum_{\a\in R^+\setminus R_P^+}\a(\eta).\]
\end{proof}

\begin{lemma}\label{3blemma}
Let $H$ be a subgroup of $G$. Suppose $\G$ has an $H$-action and $f$ is $H$-equivariant. Then for any $\eta\in\Q/\Q_P$, the stack $\MM(f,\eta)$ has a natural $H$-action such that $\ev$ is $H$-equivariant.
\end{lemma}
\begin{proof}
This follows immediately from Lemma \ref{3alemma}.
\end{proof}

\begin{definition} \label{3bmoduli} $~$
\begin{enumerate}
\item For any $\wl\in W_{af}^-$, define $\MMw$ to be the moduli space $\MM(f,\eta)$ in Definition \ref{3bmodulif} by taking $f=f_{\ag,\wl}$ (Definition \ref{2cdef}).

\item For any $\mu\in\Q$, define $\MMm$ to be the moduli space $\MM(f,\eta)$ in Definition \ref{3bmodulif} by taking $f= t^{\mu}$ (point map).
\end{enumerate}
\end{definition}

By Lemma \ref{3blemma}, $\MMw$ (resp. $\MMm$) has a natural $B^-$-action (resp. $T$-action) such that $\ev$ is $B^-$-equivariant (resp. $T$-equivariant). 
%%%%%%%%%%%%%%%%%%%%%%%%%%%%%%%%%%%
%%%%%%%%%%%%%%%%%%%%%%%%%%%%%%%%%%%
\subsection{Construction of the Savelyev-Seidel homomorphism} \label{3c}

\begin{definition}\label{3cdef}  Define an $\RRR$-linear map
\[
\begin{array}{cccc}
\Phi_{SS}: & H^T_{-\bl}(\ag) &\ra& QH_T^{\bl}(G/P)[q_i^{-1}|~i\in I\setminus I_P]\\ [0.5em]
&\ascl_{\wl} &\mapsto & \ds \sum_{v\in W^P}\sum_{\eta\in\Q/\Q_P} q^{\eta} \left(\int_{[\MMw]^{vir}}\ev^*\scl^v\right) \scl_v
\end{array}.
\]
\end{definition}

\begin{remark} \label{3crmk2} At this stage, we should take the coefficient ring to be $\QQ$. But we will prove at the end that we can actually take it to be $\ZZ$. See Theorem \ref{4dmain}.
\end{remark}

\begin{proposition}\label{3cring} $\Phi_{SS}$ is a graded homomorphism of $H_T^{\bl}(\pt)$-algebras
\end{proposition}

\begin{proof}
Let us first show that $\Phi_{SS}$ is graded. Let $\wl\in W_{af}^-$. Then $\ascl_{\wl}$ has degree $-2\ell(\wl)$ in $H^T_{-\bl}(\ag)$. By Lemma \ref{3bdim}, the integral $\int_{[\MMw]^{vir}}\ev^*\scl^v$ is non-zero only if 
\[ \dim G/P -\ell(v)=\ell(\wl)+\dim G/P +\sum_{\a\in R^+\setminus R^+_P}\a(\eta) .\]
By definition, $q^{\eta}$ has degree $2\sum_{\a\in R^+\setminus R^+_P}\a(\eta)$. This shows that $\Phi_{SS}(\ascl_{\wl})$ has degree $-2\ell(\wl)$ as desired.

It remains to show that $\Phi_{SS}$ is an algebra homomorphism. Define a $\fof(H_T^{\bl}(\pt))$-linear map
\begin{equation}\label{3cseidel}
\begin{array}{cccc}
\Phi_{SS}': & H^T_{-\bl}(\ag)_{loc} &\ra& QH_T^{\bl}(G/P)[q_i^{-1}|~i\in I\setminus I_P]_{loc}\\ [0.5em]
&[t^{\mu}] &\mapsto & \ds \sum_{v\in W^P}\sum_{\eta\in\Q/\Q_P} q^{\eta} \left(\int_{[\MMm]^{vir}}\ev^*\scl^v\right) \scl_v
\end{array}
\end{equation}
where the subscript $loc$ denotes the localization $-\otimes_{H_T^{\bl}(\pt)}\fof(H_T^{\bl}(\pt))$. By \eqref{2cprod} and Lemma \ref{3ciritani} below, it suffices to show 
\[ \Phi_{SS}(\ascl_{\wl}) = \Phi_{SS}'(\ascl_{\wl}) \]
for any $\wl\in W_{af}^-$. Put $\G:=\G_{\wl}$, the source of $f_{\ag,\wl}$. To simplify the exposition, assume $\G^T$ is discrete\footnote{This indeed suffices for our application because we can take $f_{\ag,\wl}$ to be a Bott-Samelson-Demazure-Hansen resolution which satisfies this assumption. See the proof of Lemma \ref{2clemma}.}. By the classical localization formula and the assumption that $f_{\ag,\wl}$ is the composition of a $T$-equivariant resolution $\G\ra \ol{\BB\cdot t^{\wll}}$ and the inclusion $\ol{\BB\cdot t^{\wll}}\hookrightarrow \ag$, we have
\[ \ascl_{\wl} = (f_{\ag,\wl})_*[\G]= \sum_{\g\in\G^T} \frac{1}{e^T(T_{\g}\G)} \localb{\mu_{\g}}\]
where $\mu_{\g}\in\Q$ satisfies $f_{\ag,\wl}\circ \g=t^{\mu_{\g}}$. (Here, $\g$ and $t^{\mu_{\g}}$ are viewed as morphisms from $\spec\CC$ to $\G$ and $\ag$ respectively.) Put $\MM:=\MMw$ and $\MM_{\g}:=\MM(f_{\ag,\wl}\circ\g,\eta)\simeq \MM(\mu_{\g},\eta)$. Let $\{F_{\g,j}\}_{j\in J_{\g}}$ be the set of components of the fixed-point substack $\MM_{\g}^T$. We have 
\[ \MM^T=\bigcup_{\g\in \G^T}\MM_{\g}^T=\bigcup_{\g\in\G^T}\bigcup_{j\in J_{\g}} F_{\g,j}.\]
Applying the virtual localization formula \cite{GP} twice, we get 
\[[\MM]^{vir}  = \sum_{\g\in \G^T} \sum_{j\in J_{\g}} \frac{[F_{\g,j}]^{vir}}{e^T(N^{vir}_{F_{\g,j}/\MM})} 
 = \sum_{\g\in \G^T} \frac{1}{e^T(T_{\g} \G)}\left( \sum_{j\in J_{\g}} \frac{[F_{\g,j}]^{vir}}{e^T(N^{vir}_{F_{\g,j}/\MM_{\g}})}  \right) 
= \sum_{\g\in \G^T} \frac{1}{e^T(T_{\g} \G)} [\MM_{\g}]^{vir}.\]
It follows that 
\begin{align*}
\Phi_{SS}(\ascl_{\wl}) &= \sum_{v\in W^P}\sum_{\eta\in\Q/\Q_P} q^{\eta} \left(\int_{[\MMw]^{vir}}\ev^*\scl^v\right) \scl_v\\
&= \sum_{v\in W^P}\sum_{\eta\in\Q/\Q_P} \sum_{\g\in\G^T} \frac{q^{\eta}}{ e^T(T_{\g}\G)} \left(\int_{[\MM_{\g}]^{vir}}\ev^*\scl^v\right) \scl_v\\
&= \sum_{\g\in\G^T} \frac{1}{ e^T(T_{\g}\G)}\sum_{v\in W^P}\sum_{\eta\in\Q/\Q_P} q^{\eta} \left(\int_{[\MM(\mu_{\g},\eta)]^{vir}}\ev^*\scl^v\right) \scl_v\\
&= \sum_{\g\in\G^T} \frac{1}{ e^T(T_{\g}\G)} \Phi_{SS}'([t^{\mu_{\g}}])\\
&= \Phi_{SS}'\left( \sum_{\g\in\G^T} \frac{1}{e^T(T_{\g}\G)} \localb{\mu_{\g}} \right) = \Phi_{SS}'(\ascl_{\wl})
\end{align*}
as desired.
\end{proof}

\begin{remark} 
In the proof of Proposition \ref{3cring}, we have used the fact that $\pi_*[X]=[Y]\in H^T_{BM,\dim_{\RR}Y}(Y)$ for any $T$-equivariant proper birational morphism $\pi:X\ra Y$ between possibly singular $T$-varieties over $\CC$. For reader's convenience, we provide a proof. 

Let us first deal with the non-equivariant case. Since $\pi$ is proper, the pushforward map $\pi_*:A_{\bl}(X)\ra A_{\bl}(Y)$ between Chow groups exists. Since $\pi$ is birational, it has degree one, and hence $\pi_*[X]=[Y]$ where $[X]\in A_n(X)$ and $[Y]\in A_n(Y)$ ($n:=\dim_{\CC}X=\dim_{\CC}Y$) are the fundamental cycles. The desired equality (in Borel-Moore homology) now follows from this equality, the existence of the cycle map $c\ell:A_{\bl}(-)\ra H_{BM,2\bl}(-)$ and the fact that $c\ell$ commutes with $\pi_*$. See \cite[Chapter 17]{AF} or \cite{Fulton} for more details.

For the equivariant case, apply the above result to the morphism $X\times^T U\ra Y\times^T U$ for a suitable finite dimensional approximation $U\ra U/T$ of the classifying bundle $ET\ra BT$. See \cite[Section 2.2]{EG} for more details.
\end{remark}

\begin{lemma}\label{3ciritani}
The map $\Phi_{SS}'$ defined in \eqref{3cseidel} satisfies
\begin{equation}\label{3ceqn2}
\Phi_{SS}'(\localb{\mu_1+\mu_2})= \Phi_{SS}'(\localb{\mu_1})\star\Phi_{SS}'(\localb{\mu_2})
\end{equation}
for any $\mu_1, \mu_2\in\Q$.
\end{lemma}
\begin{proof}
Notice that each $\Phi_{SS}'(\localb{\mu})$ is a $T$-equivariant Seidel element. Seidel elements are originally introduced by Seidel in \cite{Seidel}. Their $T$-equivariant generalizations are introduced in \cite{BMO, Iritani, MO, OP} in algebraic geometry and in \cite{Mak, LJ} in symplectic geometry.

Consider the one $S_{\mu}(0):=S_{\mu}(\tau)|_{\tau = 0}$ defined by Iritani in \cite[Definition 3.17]{Iritani}. (More precisely, what he defined are $T$-equivariant big Seidel elements. Since we are dealing with $T$-equivariant small Seidel elements, we put $\tau =0$.) In terms of our notations, we have
\[ S_{\mu}(0):= \sum_{v\in W^P}\sum_{\eta\in\Q/\Q_P} q^{\eta - \degg([u^{min}_{\mu}])} \left(\int_{[\MMm]^{vir}}\ev^*\scl^v\right) \scl_v= q^{-\degg([u^{min}_{\mu}])}\Phi_{SS}'(\localb{\mu})\]
where $u^{min}_{\mu}$ is a minimal section of $\fib{t^{\mu}}$ which is defined between Lemma 3.5 and Lemma 3.6 in \textit{op. cit.} and $\degg$ is the function defined in Definition \ref{3afunction} in the present paper.

By the discussion following \cite[Definition 3.17]{Iritani}, we have 
\[ q^{\degg( [u^{min}_{\mu_1+\mu_2}] - [u^{min}_{\mu_1}]\#[u^{min}_{\mu_2}] )} S_{\mu_1+\mu_2}(0) = S_{\mu_1}(0)\star S_{\mu_2}(0)\]
where $[u^{min}_{\mu_1}]\#[u^{min}_{\mu_2}]$ is the section class of $\fib{t^{\mu_1+\mu_2}}$ obtained by gluing the sections $u^{min}_{\mu_1}$ and $u^{min}_{\mu_2}$ through the following ``degeneration'' (see the proof of \cite[Corollary 3.16]{Iritani})
\[ \fibb_{\mu_1,\mu_2} := \left( (\AA^2_{a_1,a_2}\setminus 0)\times (\AA^2_{b_1,b_2}\setminus 0) \times G/P \right)/\mathbb{G}_m\times\mathbb{G}_m \]
Here the $\mathbb{G}_m\times\mathbb{G}_m$-action is defined by 
\[ (z_1,z_2)\cdot ((a_1,a_2),(b_1,b_2),y) := ((z_1^{-1}a_1,z_1^{-1}a_2),(z_2^{-1}b_1,z_2^{-1}b_2),\mu_1(z_1)\mu_2(z_2)\cdot y).\] 
(We call $\fibb_{\mu_1,\mu_2}$ a degeneration because it is a $G/P$-bundle over $\PP^1\times\PP^1$ and satisfies
\[ \fibb_{\mu_1,\mu_2}|_{\PP^1\times [1:0]} \simeq \fib{t^{\mu_1}}, \quad \fibb_{\mu_1,\mu_2}|_{\PP^1\times [0:1]} \simeq \fib{t^{\mu_2}}\quad\text{and}\quad \fibb_{\mu_1,\mu_2}|_{\Delta} \simeq \fib{t^{\mu_1+\mu_2}}\]
where $\Delta\subset \PP^1\times \PP^1$ is the diagonal.)

The equality \eqref{3ceqn2} will be proved if we can show
\[\degg([u^{min}_{\mu_1}]\#[u^{min}_{\mu_2}]) = \degg([u^{min}_{\mu_1}]) +\degg([u^{min}_{\mu_2}]) . \]
This follows from the observation that for each $\rho\in (\Q/\Q_P)^*$, the line bundle
\[ \left( (\AA^2_{a_1,a_2}\setminus 0)\times (\AA^2_{b_1,b_2}\setminus 0) \times L_{\rho} \right)/\mathbb{G}_m\times\mathbb{G}_m \]
on $\fibb_{\mu_1,\mu_2}$ restricts to $\LL_{\rho}$ over $\fibb_{\mu_1,\mu_2}|_{\PP^1\times [1:0]} $, $\fibb_{\mu_1,\mu_2}|_{\PP^1\times [0:1]} $ and $\fibb_{\mu_1,\mu_2}|_{\Delta}$.
\end{proof}

\begin{remark} The author of the present paper did not know Iritani's result until he read a paper of Gonz\'{a}lez, Mak and Pomerleano \cite{Mak}. In the original version, Lemma \ref{3ciritani} was proved using Li's degeneration formula \cite{Li1, Li2}. The degeneration used was essentially the fiber bundle $\fibb_{\mu_1,\mu_2}$ constructed by Iritani. Notice however that  Iritani's proof does not rely on the degeneration formula but virtual localization formula.
\end{remark}
%%%%%%%%%%%%%%%%%%%%%%%%%%%%%%%%%%%
%%%%%%%%%%%%%%%%%%%%%%%%%%%%%%%%%%%
%%%%%%%%%%%%%%%%%%%%%%%%%%%%%%%%%%%
%%%%%%%%%%%%%%%%%%%%%%%%%%%%%%%%%%%
%%%%%%%%%%%%%%%%%%%%%%%%%%%%%%%%%%%
%%%%%%%%%%%%%%%%%%%%%%%%%%%%%%%%%%%
\section{Proof of main result}\label{4}
%nobadbox
\subsection{$T$-invariant sections} \label{4a}
Let $\mu\in\Q$. Recall $\fib{t^{\mu}}$ is the $G/P$-bundle $\fib{f}$ where we take $f=t^{\mu}$. By definition, we have
\begin{equation}\label{4afib}
\fib{t^{\mu}}\simeq \left(\AA^1_z\times G/P\times\{0,\infty\}\right)/_{(z,y,0)~\sim~ (z^{-1},\mu(z)\cdot y,\infty)}.
\end{equation}
\noindent Every $v\in W^P$ gives rise to a $T$-invariant section $u_{\mu,v}$ of $\fib{t^{\mu}}$ defined by
\[ u_{\mu,v}([z_1:z_2]) := [z_1/z_2,y_v,0]=[z_2/z_1,y_v,\infty ],~\quad [z_1:z_2]\in\PP^1.\]
It is easy to see that all $T$-invariant sections of $\fib{t^{\mu}}$ arise in this way.
 
Let $v\in W^P$. By linearizing the $G$-action on $G/P$ at $y_v$, we obtain an isomorphism
\[ T_{y_v}(G/P)\simeq \bigoplus_{\a\in -v(R^+\setminus R_P^+)}\gg_{\a}\]
of $T$-modules.

\begin{lemma} \label{4asectiondeg} Let $\TT^{vert}$ be the vertical tangent bundle of the $G/P$-bundle $\fib{t^{\mu}}\ra\PP^1$. Then $u_{\mu,v}^*\TT^{vert}$ is defined by the transition matrix
\[ A(z):= \sum_{\a\in -v(R^+\setminus R_P^+)} z^{\a(\mu)}\id_{\gg_{\a}}\in\eend(T_{y_v}(G/P))[z,z^{-1}]. \]
In particular, we have
\[u_{\mu,v}^*\TT^{vert}\simeq \bigoplus_{\a\in -v(R^+\setminus R_P^+)}\OO_{\PP^1}(-\a(\mu)).\]
\end{lemma}
\begin{proof} This follows from the explicit construction \eqref{4afib} of $\fib{t^{\mu}}$.
\end{proof}

Recall the function $\degg$ defined in Definition \ref{3afunction}.
\begin{lemma}\label{4asectionclass} For any $\mu\in\Q$ and $v\in W^P$, we have $\degg([u_{\mu,v}]) = v^{-1}(\mu) + \Q_P\in \Q/\Q_P$. 
\end{lemma}
\begin{proof}
Write $\degg([u_{\mu,v}]) = \eta + \Q_P$. Let $\rho\in (\Q/\Q_P)^*$. By definition, $\rho(\eta)$ is the degree of the line bundle $u_{\mu,v}^*\LL_{\rho}$. From the explicit construction \eqref{4afib} of $\fib{t^{\mu}}$ and the definition of $\LL_{\rho}$, we see that $u_{\mu,v}^*\LL_{\rho}$ is defined by the transition matrix $-\rho(v^{-1}(\mu))$. It follows that the degree is equal to $\rho(v^{-1}(\mu))$. Since $\rho$ is arbitrary, the result follows. 
\end{proof}
%%%%%%%%%%%%%%%%%%%%%%%%%%%%%%%%%%%
%%%%%%%%%%%%%%%%%%%%%%%%%%%%%%%%%%%
\subsection{Regularity of the moduli} \label{4b}
Recall the key reason for $\MM_{0,n}(G/P,\b)$ to be regular is that $G/P$ is \text{convex}, i.e.
\begin{equation}\label{4bh1}
H^1(C;u^*\TT_{G/P})=0
\end{equation}
for any morphism $u:C\ra G/P$ where $C$ is a genus zero nodal curve. Surprisingly, $\fib{f_{\ag,\wl}}$ also satisfies this property, provided the morphisms in question represent section classes. The goal of this subsection is to prove this fact. First, we show that it suffices to verify the analogue of \eqref{4bh1} for a smaller class of $u$. In what follows, $C$ always denotes a genus zero nodal curve.
\begin{definition}\label{4bTfix}
Let $X$ be a variety with a $T$-action. A morphism $u:C\ra X$ is said to be \text{$T$-invariant} if for any $t\in T$ there exists an automorphism $\phi:C\ra C$ such that $t\cdot u=u\circ\phi$. 
\end{definition}

\begin{lemma}\label{4blemma} 
Let $X$ be a smooth projective variety with a $T$-action and $\b\in H_2(X)$. Suppose for any $T$-invariant morphism $u:C\ra X$ representing $\b$, we have $H^1(C;u^*\TT_X)=0$. Then the same is true for any morphism representing $\b$.
\end{lemma}
\begin{proof}
For a given morphism, choose $n\in\ZZ_{\geqslant 0}$ such that it becomes stable after adding $n$ marked points to its domain. Let $\ol{M}:=\ol{M}_{0,n}(X,\b)$ be the coarse moduli space of stable maps to $X$ with $n$ marked points and representing $\b$. This space is constructed and proved to be projective in \cite[Theorem 1]{FP}. Denote by $V$ the set of $[u]\in\ol{M}$ such that $H^1(C;u^*\TT_X)=0$. We have to prove $\ol{M}=V$. Notice that $T$ preserves $V$, and hence its complement $\ol{M}\setminus V$. Let us assume for a while $V$ is open so that $\ol{M}\setminus V$ is closed. Suppose $\ol{M}\setminus V\ne\emptyset$. By Borel fixed-point theorem, $\ol{M}\setminus V$ contains a $T$-fixed point $[u_0]$. Then $u_0$ is $T$-invariant and $H^1(C_0;u_0^*\TT_X)\ne 0$, in contradiction to our assumption stated in the lemma. Therefore, $\ol{M}=V$, as desired.

It remains to verify that $V$ is open. Recall \cite[Section 3 \& 4]{FP} $\ol{M}$ is a union of open subschemes each of which is a finite group quotient of the fine moduli $U$ of stable maps to $X$ with stable domains, representing $\b$ and satisfying a condition depending on a fixed set of generic Cartier divisors on $X$. For each $U$, consider its universal family $\pi:\mathcal{C}\ra U$ and evaluation map $\ev:\mathcal{C}\ra X$. Since $\pi$ is flat and $\ev^*\TT_X$ is locally free, the set $U'$ of $x\in U$ for which $H^1(\mathcal{C}_x;\ev^*\TT_X|_{\mathcal{C}_x})=0$ is open, by the semi-continuity theorem. Then $U'$ descends to an open subset $U''$ of $V$. The proof is complete by varying $U$ and taking the union of $U''$. 
\end{proof}

\begin{proposition} \label{4bmain} Let $\G$ be a smooth projective variety and $f:\G\ra\ag$ a morphism which is $T$-good (see Definition \ref{2cgood}). Then for any morphism $u:C\ra \fib{f}$ which represents a section class of $\fib{f}$, we have $H^1(C;u^*\TT_{\fib{f}})=0$.
\end{proposition}
\begin{proof}
Since $f$ is $T$-good, $\fib{f}$ has a $T$-action by Lemma \ref{3alemma}, and hence, by Lemma \ref{4blemma}, we may assume $u$ is $T$-invariant.

Consider the composition
\[ \pr_{\G}\circ\pi_f\circ u:C\ra \fib{f}\ra \PP^1\times \G \ra \G.\]
Since $u$ represents a section class, we have $(\pr_{\G}\circ\pi_f\circ u)_*[C]=0$. But $\G$ is projective so $\pr_{\G}\circ\pi_f\circ u$ is constant, and hence there exists a factorization
\[ u:C\xrightarrow{u'}\fib{f\circ\g}\xhookrightarrow{\iota}\fib{f}\]
for some morphisms $\g:\spec\CC\ra\G$ and $u':C\ra\fib{f\circ\g}$ where $\iota$ is the canonical inclusion.

Consider next the composition
\[ \pr_{\PP^1}\circ\pi_{f\circ \g}\circ u':C\ra \fib{f\circ\g}\ra \PP^1\times \spec\CC \xrightarrow{\sim} \PP^1.\]
Since $u$ represents a section class, we have $(\pr_{\PP^1}\circ\pi_{f\circ \g}\circ u')_*[C]=[\PP^1]$. It follows that we can write $C=C_0\cup C_1$ where $C_0\simeq \PP^1$ is an irreducible component of $C$ and $C_1$ is the union of the other irreducible components, such that $u'|_{C_0}$ is a section of $\fib{f\circ\g}$ after reparametrizing $C_0$ and $u'|_{C_1}$ factors through a finite union of the fibers of $\pi_{f\circ \g}$.

Let us first deal with the case where $C_1$ is absent. In what follows, we will identify $C_0$ with $\PP^1$ and assume $u'=u'|_{C_0=\PP^1}$ is a section of $\fib{f\circ\g}$. Define $\FF:=u^*\TT^{vert}_{\pr_{\PP^1}\circ\pi_f}$ where $\TT^{vert}_{\pr_{\PP^1}\circ\pi_f}$ is the vertical tangent bundle of the fiber bundle:
\[ \pr_{\PP^1}\circ\pi_f :\fib{f}\ra \PP^1\times\G\ra \PP^1. \]
Since $H^1(\PP^1;\TT_{\PP^1})=0$, it suffices to verify $H^1(\PP^1;\FF)=0$. Define $\FF':= u^*\TT^{vert}_{\pi_f}$ where $\TT^{vert}_{\pi_f}$ is the vertical tangent bundle of $\pi_f$. We have an exact sequence of coherent sheaves over $C_0=\PP^1$:
\begin{equation}\label{4bexact}
 0\ra\FF'\ra\FF\ra T_{\g}\G\otimes_{\CC}\OO_{\PP^1}\ra 0
\end{equation}
where the morphism $\FF\ra T_{\g}\G\otimes_{\CC}\OO_{\PP^1}$ is given by the projection. By looking at the associated long exact sequence, it suffices to show
\begin{equation}\label{4bneed}
 \dim H^1(\PP^1;\FF') \leqslant \dim\coker(H^0(\PP^1;\FF)\ra T_{\g}\G). 
\end{equation}

Let us look at $\FF'$ closely. Since $u$ is $T$-invariant, we have $\g\in \G^T$, and so $f\circ\g=t^{\mu}$ for some $\mu\in\Q$. By the discussion in Section \ref{4a}, we have $u'=u_{\mu,v}$ for some $v\in W^P$ (after identifying $\fib{f\circ\g}$ with $\fib{t^{\mu}}$). Put $R_v:=-v(R^+\setminus R^+_P)$. Then by Lemma \ref{4asectiondeg}, $\FF'$ is defined by the transition matrix
\begin{equation}\label{4badded1}
 A(z):= \sum_{\a\in R_v} z^{\a(\mu)}\id_{\gg_{\a}}\in\eend(T_{y_v}(G/P))[z,z^{-1}]
\end{equation}
and in particular $\FF'\simeq\bigoplus_{\a\in R_v}\OO_{\PP^1}(-\a(\mu))$. (Recall we have identified $T_{y_v}(G/P)$ with $\bigoplus_{\a\in R_v}\gg_{\a}$ via the linearization of the $G$-action on $G/P$ at $y_v$.) Since for any $m\in\ZZ$ 
\[ \dim H^1(\PP^1,\OO(m))=\#\{k\in\ZZ|-m>k>0\}, \]
it follows that 
\begin{equation}\label{4badded2}
\dim H^1(\PP^1;\FF') = \#\{ (\a,k)\in R_v\times\ZZ|~\a(\mu)>k>0\} .
\end{equation}

Let us now look at $\FF$. By \eqref{4bexact}, $\FF$ is defined by a transition matrix of the form 
\[ 
\begin{bmatrix}
A(z)& B(z)\\
0&\id
\end{bmatrix}
\]
for some $B(z)\in \ehom(T_{\g}\G,T_{y_v}(G/P))[z,z^{-1}]$. It follows that every element of $H^0(\PP^1;\FF)$ is given by a pair of polynomial maps
\[ u_1:\AA^1\ra T_{y_v}(G/P)\quad\text{ and }\quad u_2:\AA^1 \ra T_{\g}\G\]
such that the Laurent polynomials
\[ A(z)u_1(z)+B(z)u_2(z)\quad \text{ and }\quad u_2(z)\]
are polynomials in $z^{-1}$. It is clear that $u_2(z)\equiv\zeta$ for some constant $\zeta\in T_{\g}\G$. Write $u_1(z)=\sum_{\a\in R_v} u_{1,\a}(z)$ where $u_{1,\a}:\AA^1\ra\gg_{\a}$; and $B(z)=\sum_{\a\in R_v}\sum_{k\in\ZZ}z^k B_{\a,k}$ where $B_{\a,k}:T_{\g}\G\ra \gg_{\a}$ is linear. The above condition for $A(z)u_1(z)+B(z)u_2(z)$ is equivalent, given $u_2(z)\equiv\zeta$, to the one that for any $\a\in R_v$, the Laurent polynomial
\begin{equation}\label{4bexpression1}
z^{\a(\mu)}u_{1,\a}(z)+\sum_{k\in\ZZ}z^kB_{\a,k}(\zeta)
\end{equation}
is a polynomial in $z^{-1}$. Since $z^kB_{\a,k}(\zeta)$ cannot cancel any term from $z^{\a(\mu)}u_{1,\a}(z)$ for any $k$ such that $\a(\mu)>k$, the above condition for \eqref{4bexpression1} implies that for any $\a\in R_v$ and $\a(\mu)>k>0$, we have $B_{\a,k}(\zeta)=0$. 

Define $h$ to be the composition
\begin{equation}\label{4bexpression15}
T_{\g}\G\xrightarrow{B(z)} T_{y_v}(G/P)[z,z^{-1}]\simeq \bigoplus_{\substack{\a\in R_v\\ k\in\ZZ}}z^k\gg_{\a}\ra \bigoplus_{\substack{ \a\in R_v\\ \a(\mu)>k>0}}z^k\gg_{\a}
\end{equation}
where the last arrow is the canonical projection. The discussion in the last paragraph implies that the composition
\begin{equation}\label{4bexpression2}
H^0(\PP^1;\FF)\ra T_{\g}\G\xrightarrow{h} \bigoplus_{\substack{ \a\in R_v\\ \a(\mu)>k>0}}z^k\gg_{\a}
\end{equation}
is zero. By Lemma \ref{4bsurj} below which says that $h$ is surjective, we have
\begin{equation}\label{4bexpression3}
\#\{(\a,k)\in R_v\times\ZZ|~\a(\mu)>k>0\}=\dim(\text{RHS of }\eqref{4bexpression2})\leqslant\dim\coker(H^0(\PP^1;\FF)\ra T_{\g}\G).
\end{equation}
But the LHS of \eqref{4bexpression3} is equal to $\dim H^1(\PP^1;\FF')$ by \eqref{4badded2}. This gives inequality \eqref{4bneed}. Hence the proof for the case where $C_1$ is absent is complete.

Finally, we deal with the general case. By the normalization sequence (e.g. \cite{Mirror}), it suffices to show 
\begin{enumerate}
\item $H^1(C_0; u^*\TT_{\fib{f}}|_{C_0})=H^1(C_1; u^*\TT_{\fib{f}}|_{C_1})=0$; and 
\item the evaluation map $H^0(C_1; u^*\TT_{\fib{f}}|_{C_1})\ra \bigoplus_i T_{u(p_i)}\fib{f}$ at the intersection points $\{p_i\}$ of $C_0$ and $C_1$ is surjective.
\end{enumerate}
We have proved $H^1(C_0; u^*\TT_{\fib{f}}|_{C_0})=0$. Observe that $u^*\TT_{\fib{f}}|_{C_1}$ is an extension of a trivial bundle by $(u|_{C_1})^*\TT^{vert}_{\pi_f}$. The rest of the statements then follow from the well-known fact that $\TT_{G/P}$ is globally generated. The proof of Proposition \ref{4bmain} is complete. 
\end{proof}

\begin{lemma}\label{4bsurj} 
The map $h$ defined in \eqref{4bexpression15} is surjective.
\end{lemma}
\begin{proof}
Let $\a\in R_v$ and $\a(\mu)>k>0$. Pick a non-zero vector $X_{\a}\in\gg_{\a}$. Define $r_{\a,k}:\AA^1_s\ra \G$ by $s\mapsto \exp(sz^kX_{\a})\cdot \g$ where the action is the given $U_{\a,k}$-action on $\G$. The surjectivity of $h$ follows if we can show that $h$ sends $v:= D_{s=0}r_{\a,k}(1)\in T_{\g}\G$ to $z^kX_{\a}\in z^k\gg_{\a}$.

Consider the $G/P$-bundle $\fib{f\circ r_{\a,k}}$ over $\PP^1\times\AA^1_s$. Notice that $u$ naturally factors through a morphism $u'':C_0=\PP^1\ra\fib{f\circ r_{\a,k}}$. Since $f$ is $T$-good and in particular $U_{\a,k}$-equivariant, $f\circ r_{\a,k}$ is equal to the morphism $s\mapsto \exp(sz^kX_{\a})\cdot t^{\mu}$. By the definition of the $U_{\a,k}$-action on $\ag$ (see \eqref{2caction}), we have
\[ \fib{f\circ r_{\a,k}} \simeq \left(\AA^1_z\times \AA^1_s\times G/P\times \{0,\infty\}\right)/_{(z,s,y,0)~\sim~(z^{-1},s,\exp(sz^kX_{\a})\mu(z)\cdot y,\infty)}.\]
From this explicit construction we see that the vector bundle $(u'')^*\TT^{vert}_{\pr_{\PP^1}\circ\pi_{f\circ r_{\a,k}}}$ is defined by a transition matrix of the form
\[ 
\begin{bmatrix}
A(z)& z^kX_{\a}\\
0&\id
\end{bmatrix}
\]
where $A(z)$ is the same as the one defined in \eqref{4badded1}. Since the transition matrix 
\[ 
\begin{bmatrix}
A(z)& B(z)v\\
0&\id
\end{bmatrix}
\]
also defines the same vector bundle, these two matrices differ by a gauge transformation. A straightforward computation shows that the difference $B(z)v-z^kX_{\a}$ lies in the sum of $z^{k'}\gg_{\a'}$ with $\a'\in R_v$ and $k'\leqslant 0$ or $k'\geqslant \a(\mu)$. Since $\a(\mu)>k>0$, we have $h(v)=z^kX_{\a}$ as desired.
\end{proof}

Let $\wl\in W_{af}^-$. Recall 
\[ f_{\ag,\wl}:\G_{\wl}\ra \ag\]
is the $B^-$-good morphism fixed in Definition \ref{2cdef}. Clearly, it is $T$-good. It follows that the condition in Proposition \ref{4bmain} is satisfied, and hence $\MMw$ is regular for any $\eta\in\Q/\Q_P$. Moreover, since $f_{\ag,\wl}$ is $B^-$-equivariant, it follows that by Lemma \ref{3blemma} $\MMw$ has a $B^-$-action and $\ev:\MMw\ra G/P$ is $B^-$-equivariant.

Now let $v\in W^P$. Recall 
\[ f_{G/P,v} :\G_v\ra G/P \]
is the $B^+$-equivariant morphism fixed in Definition \ref{2bdef}. By Lemma \ref{2blemma}, $f_{G/P,v}$ is transverse to $\ev:\MMw\ra G/P$, i.e. the sum of the images of the tangent maps of these morphisms is equal to the tangent space of the common target. It follows that the stack 
\[  \MM(\wl,v,\eta):= \MMw\times_{(\ev,f_{G/P,v})} \G_v\]
is regular. Notice that there is still a $T$-action on $\MM(\wl,v,\eta)$, since $T=B^-\cap B^+$.

\begin{lemma}\label{4bdim}
Suppose $\MM(\wl,v,\eta)\ne\emptyset$. The dimension of $\MM(\wl,v,\eta)$ is equal to $\ell(\wl)+\ell(v)+\sum_{\a\in R^+\setminus R^+_P}\a(\eta)$.
\end{lemma}  
\begin{proof}
By Lemma \ref{3bdim}, the virtual dimension of $\MMw$ is equal to $\ell(\wl)+\dim G/P +\sum_{\a\in R^+\setminus R^+_P}\a(\eta)$. It follows that the virtual dimension of $\MM(\wl,v,\eta)$ is equal to 
\begin{align*}
& \left( \ell(\wl)+\dim G/P +\sum_{\a\in R^+\setminus R^+_P}\a(\eta)\right) + \ell(v) - \dim G/P \\
=&~ \ell(\wl)+\ell(v)+\sum_{\a\in R^+\setminus R^+_P}\a(\eta).
\end{align*}
Since  $\MM(\wl,v,\eta)$ is regular, its dimension is equal to its virtual dimension. The proof is complete.
\end{proof}
%%%%%%%%%%%%%%%%%%%%%%%%%%%%%%%%%%%
%%%%%%%%%%%%%%%%%%%%%%%%%%%%%%%%%%%
\subsection{Zero-dimensional components} \label{4c} Let $\wl\in W_{af}^-$, $v\in W^P$ and $\eta\in\Q/\Q_P$. Put $\MM:=\MM(\wl,v,\eta)$, the stack defined at the end of Section \ref{4b}.
\begin{proposition}\label{4cmain} The stack $\MM$ is non-empty and zero-dimensional if and only if $v\in wW_P$, $\eta=\l+\Q_P$ and the following set of conditions, which we denote by $C(\wl)$, holds:
\[\left\{
\begin{array}{rcl}
\a\in (-w R_P^+)\cap R^+ &\Longrightarrow& \a(w(\l))=1\\
\a\in (-w R_P^+)\cap (-R^+) &\Longrightarrow& \a(w(\l))=0
\end{array}
\right. .\]
In this case, $\MM$ is a one-point stack with trivial stabilizer.
\end{proposition}
\begin{proof}
Suppose $\MM\ne\emptyset$ and $\dim\MM=0$. Notice that the boundary of $\MM$ is stratified by the moduli spaces of stable maps satisfying the same conditions as those imposed on points of $\MM$, plus the condition that their domain curves are reducible and have fixed combinatorial types. Arguing as before, we conclude that these strata are smooth and of expected dimension. Since $\dim\MM=0$, they are empty, and hence every point of $\MM$ is represented by a stable map $u$ to $\fib{f_{\ag,\wl}}$ which factors through a section $u'$ of $\fib{f_{\ag,\wl}\circ \g}$ for some $\g:\spec\CC\ra\G_{\wl}$. This section is necessarily $T$-invariant because $\MM$ is zero-dimensional and has a $T$-action. It follows that $\g\in\G_{\wl}^T$, and hence $f_{\ag,\wl}\circ \g=t^{\mu_{\g}}$ for some $\mu_{\g}\in\Q$. Thus we have $u'=u_{\mu_{\g},v'}$ for some $v'\in W^P$, after identifying $\fib{f_{\ag,\wl}\circ \g}$ with $\fib{t^{\mu}}$.

Let us show $\mu_{\g}=\wll$. Let $w't_{\l'}\in W_{af}^-$ be the unique element such that $\mu_{\g}=w'(\l')$. Since $t^{\mu_{\g}}\in\ol{\BB\cdot t^{\wll}}$, we have $\ell(w't_{\l'})\leqslant\ell(\wl)$, and the equality holds if and only if $\wl=w't_{\l'}$. Observe that the section $u_{\mu_{\g},v'}$ also represents a point of $\MM':=\MM(w't_{\l'},v,\eta)$. It follows that $\MM'\ne\emptyset$, and hence, by the regularity, we have $\dim\MM'\geqslant 0$. But by Lemma \ref{4bdim},
\[ 0=\dim\MM=\ell(\wl)+\ell(v)+ \sum_{\a\in R^+\setminus R^+_P}\a(\eta) \geqslant \ell(w't_{\l'})+\ell(v)+ \sum_{\a\in R^+\setminus R^+_P}\a(\eta)=\dim\MM'\geqslant 0. \]
It follows that $\ell(\wl)=\ell(w't_{\l'})$, and hence $\wl=w't_{\l'}$ as desired.

By a similar argument, we have $v'=v$.

To finish the proof, we need the following explicit formulae for the terms $\ell(\wl)$, $\ell(v)$ and $\sum_{\a\in R^+\setminus R^+_P}\a(\eta)=\langle [\PP^1], c_1(u_{w(\l),v}^*\TT^{vert})\rangle $ where $\TT^{vert}$ is the vertical tangent bundle of the fiber bundle $\fib{t^{w(\l)}}\ra\PP^1$. To formulate them, pick a regular dominant element $a\in \hh_{\RR}:=\Q\otimes_{\ZZ}\RR$ which is sufficiently close to the origin and a dominant element $b\in\hh_{\RR}$ which determines the parabolic type of $P$, i.e. $\a_i(b)=0$ if $\a_i\in R^+_P$ and $\a_i(b)>0$ otherwise. We have
\begin{align} \label{4ca}
\ell(\wl) &= \sum_{\a(w(\l)-a)>0}\lfloor \a(w(\l)-a)\rfloor\\
\ell(v) &= -\sum_{\a(v\cdot b)<0}\lfloor \a(-a)\rfloor \nonumber\\
\langle [\PP^1],c_1(u_{w(\l),v}^*\TT^{vert})\rangle &= -\sum_{\a(v\cdot b)<0} \a(w(\l)) \nonumber
\end{align}
where the summations are taken over $\a\in R$ satisfying the stated conditions. The first formula will be proved below, the second is obvious, and the last follows from Lemma \ref{4asectiondeg}. Summing up these equations and using the assumption $\dim\MM=0$, we obtain 
\[ \sum_{\a(w(\l)-a)>0}\lfloor \a(w(\l)-a)\rfloor - \sum_{\a(v\cdot b)<0}\lfloor \a(w(\l)-a)\rfloor=\dim\MM=0.\]
The last equation can be written as 
\begin{equation}\label{4cd}
\sum_{\a(w(\l)-a)>0}(1-A(\a,v))\lfloor \a(w(\l)-a)\rfloor + B(\a,v) = 0
\end{equation}
where 
\[  A(\a,v):= \left\{ 
\begin{array}{cc}
-1& \a(v\cdot b)>0\\
0& \a(v\cdot b)=0\\
1& \a(v\cdot b)<0
\end{array}
\right. \quad\text{ and }\quad 
B(\a,v) :=\left\{
\begin{array}{cc}
0& \a(v\cdot b)\leqslant 0\\
1& \a(v\cdot b)>0
\end{array}
\right. .\]
Observe that each of the summands of the LHS of \eqref{4cd} is non-negative. It follows that they are all equal to 0. This holds precisely when the following conditions are satisfied:
\[\left\{
\begin{array}{rcl}
\a\in v(R^+\setminus R_P^+)  &\Longrightarrow& \a\in wR^+\\
\a\in vR_P \cap (-wR^+)\cap  R^+ &\Longrightarrow& \a(w(\l))=1\\
\a\in vR_P \cap (-wR^+)\cap (-R^+) &\Longrightarrow& \a(w(\l))=0
\end{array}
\right. .\]
Here, we have used the assumption $\wl\in W_{af}^-$ which implies $-w(\l)+a\in w\mr{\Lambda}$ where $\mr{\Lambda}$ is the interior of the dominant chamber. Notice that the first condition is equivalent to $v\in wW_P$, and the conjunction of the other two is equivalent, given the first condition, to $C(\wl)$, since $vR_P\cap(-wR^+)=-wR_P^+$ if $v\in wW_P$. By Lemma \ref{4asectionclass} and the fact that every element of $W_P$ descends to the identity in the quotient $\Q/\Q_P$, we have $\eta=c([u_{w(\l),v}])=v^{-1}w(\l)+\Q_P=\l+\Q_P$. This proves one direction of Proposition \ref{4cmain}. The other direction is clear from the above discussion.

The last assertion follows from the above discussion and the fact that 
\[\# f_{\ag,\wl}^{-1}(t^{\wll})=1\quad\text{ and }\quad\# f_{G/P,v}^{-1}(y_v)=1 .\]
\end{proof}

\bigskip
\begin{myproof}{formula}{\eqref{4ca}} Denote by $\Delta_0$ the dominant alcove. Since  $\wl$ is a minimal length coset representative, the line segment joining $w(\l)$ and $a$ intersects the interior of $\wl(\Delta_0)$. Therefore, $\ell(\wl)$ is equal to the number of affine walls intersecting the interior of this line segment which is easily seen to be the RHS of \eqref{4ca}. 
\end{myproof}
%%%%%%%%%%%%%%%%%%%%%%%%%%%%%%%%%%%
%%%%%%%%%%%%%%%%%%%%%%%%%%%%%%%%%%%
\subsection{Final step} \label{4d}
Following \cite[Lemma 10.2]{LS}, we define $(W^P)_{af}$ to be the set of $\wl\in W_{af}$ such that
\begin{equation}\label{4dadded1}
\left\{
\begin{array}{rcl}
\a\in R_P^+\cap (-w^{-1}R^+) &\Longrightarrow& \a(\l)=-1\\
\a\in R_P^+\cap w^{-1}R^+&\Longrightarrow& \a(\l)=0
\end{array}
\right. .
\end{equation}
\begin{theorem}\label{4dmain} The $\RRR$-algebra homomorphism $\Phi_{SS}$ defined in Definition \ref{3cdef} satisfies 
\[\Phi_{SS}(\ascl_{\wl})=\left\{
\begin{array}{cc}
q^{\l+\Q_P}\scl_{\widetilde{w}}& \wl\in (W^P)_{af}\\ [.5em]
0& \text{otherwise}
\end{array}
\right.
\]
for any $\wl\in W_{af}^-$, where $\widetilde{w}\in W^P$ is the minimal length representative of the coset $wW_P$. 
\end{theorem}
\begin{proof}
Write $\Phi_{SS}(\ascl_{\wl})=\sum_{v\in W^P}\sum_{\eta\in\Q/\Q_P}q^{\eta} c_{\eta,v} \scl_v$. Since $\MM(\wl,v,\eta)$ is regular and $f_{G/P,v}$ is the composition of a $T$-equivariant resolution $\G_v\ra \ol{B^+\cdot y_v}$ and the inclusion $\ol{B^+\cdot y_v}\hookrightarrow G/P$, we have 
\[c_{\eta,v}=\int_{\MM(\wl,v,\eta)} 1 \in \RRR\]
which is zero unless $\MM(\wl,v,\eta)$ is non-empty and zero-dimensional. By Proposition  \ref{4cmain}, the last condition is equivalent to $v\in wW_P$, $\eta=\l+\Q_P$ and the condition $C(\wl)$, and in this case $c_{\eta,v}=1$. It remains to show that $C(\wl)$  is equivalent to the condition $\wl\in (W^P)_{af}$. This is proved by replacing $\a$ in \eqref{4dadded1} with $-w^{-1}\a$. 
\end{proof}
%%%%%%%%%%%%%%%%%%%%%%%%%%%%%%%%%%%
%%%%%%%%%%%%%%%%%%%%%%%%%%%%%%%%%%%
%%%%%%%%%%%%%%%%%%%%%%%%%%%%%%%%%%%
%%%%%%%%%%%%%%%%%%%%%%%%%%%%%%%%%%%
%%%%%%%%%%%%%%%%%%%%%%%%%%%%%%%%%%%
%%%%%%%%%%%%%%%%%%%%%%%%%%%%%%%%%%%

%%%%%%%%%%%%%%%%%%%%%%%%%%%%%%%%%%%
%%%%%%%%%%%%%%%%%%%%%%%%%%%%%%%%%%%
%%%%%%%%%%%%%%%%%%%%%%%%%%%%%%%%%%%
%%%%%%%%%%%%%%%%%%%%%%%%%%%%%%%%%%%
%%%%%%%%%%%%%%%%%%%%%%%%%%%%%%%%%%%
%%%%%%%%%%%%%%%%%%%%%%%%%%%%%%%%%%%

\end{document}